\documentclass[11pt]{amsart}
\usepackage{amsfonts,latexsym,amsthm,amssymb,graphicx}
\usepackage[all]{xy}
\usepackage[usenames]{color} 
\usepackage{booktabs}
\usepackage{hyperref}
\definecolor{PineGreen}{rgb}{0.0,0.47,0.44}
\definecolor{MidnightBlue}{rgb}{0.1,0.1,0.44}
\definecolor{magenta}{rgb}{1.0,0.0,1.0}
\hypersetup{
  colorlinks=true,       
  linkcolor=MidnightBlue,
  citecolor=PineGreen,   
  filecolor=magenta,     
  urlcolor=MidnightBlue  
}

\DeclareFontFamily{OT1}{rsfs}{}
\DeclareFontShape{OT1}{rsfs}{n}{it}{<-> rsfs10}{}
\DeclareMathAlphabet{\mathscr}{OT1}{rsfs}{n}{it}

\CompileMatrices

\newtheorem{theorem}{Theorem}[section]
\newtheorem{lemma}[theorem]{Lemma}
\newtheorem{corol}[theorem]{Corollary}
\newtheorem{prop}[theorem]{Proposition}
\newtheorem{claim}[theorem]{Claim}
\newtheorem*{theorem*}{Theorem}

{\theoremstyle{definition} }
{\theoremstyle{remark} \newtheorem{remark}[theorem]{Remark}
}

\numberwithin{equation}{section}

\newcommand{\Cbb}{{\mathbb{C}}}
\newcommand{\Pbb}{{\mathbb{P}}}
\newcommand{\Qbb}{{\mathbb{Q}}}
\newcommand{\cPbb}{{\check\Pbb}}
\newcommand{\Rbb}{{\mathbb{R}}}
\newcommand{\Tbb}{{\mathbb{T}}}
\newcommand{\ch}{\check h}
\newcommand{\cE}{{\mathcal E}}
\newcommand{\cI}{{\mathscr I}}
\newcommand{\cL}{{\mathscr L}}
\newcommand{\cM}{{\mathcal M}}
\newcommand{\cO}{{\mathscr O}}
\newcommand{\cP}{{\mathcal P}}
\newcommand{\us}{\underline s}
\newcommand{\ux}{\underline x}
\newcommand{\one}{1\hskip-3.5pt1}
\newcommand{\csm}{{c_{\text{SM}}}}
\newcommand{\cf}{{c_{\text{F}}}}
\newcommand{\cma}{{c_{\text{Ma}}}}

\DeclareMathOperator{\codim}{codim}
\DeclareMathOperator{\Eu}{Eu}
\DeclareMathOperator{\Edd}{EDdeg}
\DeclareMathOperator{\gEdd}{gEDdeg}
\DeclareMathOperator{\chima}{\chi_\text{Ma}}
\DeclareMathOperator{\Sym}{Sym}

\newcommand{\qede}{\hfill$\lrcorner$}

\title{The Euclidean distance degree of smooth complex projective varieties}
\author{Paolo Aluffi}
\author{Corey Harris}
\address{
Mathematics Department, 
Florida State University,
Tallahassee FL 32306, U.S.A.
}
\email{aluffi@math.fsu.edu}
\address{
Max-Planck-Institut f\"ur Mathematik in den Naturwissenschaften,
Inselstra{\ss}e 22,
04103 Leipzig, Germany
}
\email{Corey.Harris@mis.mpg.de}

\begin{document}

\begin{abstract}
We obtain several formulas for the Euclidean distance degree (ED degree) of an
arbitrary nonsingular variety in projective space: in terms of Chern and Segre classes, 
Milnor classes, Chern-Schwartz-MacPherson classes, and an extremely simple 
formula equating the Euclidean distance degree of $X$ with the Euler 
characteristic of an open subset of~$X$.
\end{abstract}

\maketitle


\section{Introduction}\label{intro}

The {\em Euclidean distance degree\/} (ED degree) of a variety $X$ is the
number of critical points of the distance function from a general
point outside of $X$. This definition, tailored to real algebraic varieties,
may be adapted for complex projective varieties, and this is the context
in which we will work in this paper. The ED degree is studied thoroughly 
in~\cite{MR3451425}, which provides a wealth of examples, results, and 
applications. In particular, \cite[Theorem~5.4]{MR3451425} states that the 
ED degree of a complex projective variety $X\subseteq\Pbb^{n-1}$ equals 
the sum of its `polar degrees', provided that the variety satisfies a technical 
condition related to its intersection with the {\em isotropic quadric,\/} i.e., 
the quadric $Q$ with equation $x_1^2+\cdots +x_n^2=0$. 
As a consequence, a formula is obtained (\cite[Theorem~5.8]{MR3451425}) 
computing the Euclidean distance degree of a {\em nonsingular\/} variety 
$X$, assuming that $X$ intersects $Q$ transversally, i.e., under the assumption 
that $Q\cap X$ is a nonsingular hypersurface of $X$.
This number is a certain combination of the degrees of the components of the 
{\em Chern class\/} of $X$ (see~\eqref{eq:gEdd}); we call this number the 
`generic Euclidean distance degree' of $X$, $\gEdd(X)$, since it equals the 
Euclidean distance degree of a general translate of $X$.

There are several directions in which this formula could be generalized.
For example, the hypothesis of nonsingularity on $X$ could be relaxed; it is then
understood that the role of the Chern class of $X$ is taken by the so-called
{\em Chern-Mather\/} class of $X$, one of several generalizations of the notion
of Chern class to possibly singular $X$. The resulting formula (see e.g., 
\cite[Proposition~2.9]{produa}) gives the generic ED degree of an arbitrarily 
singular variety~$X$. In a different direction, one
could maintain the nonsingularity hypothesis, but attempt to dispose of any
requirement regarding the relative position of $Q$ and $X$, and aim at
computing the `actual' ED degree of $X$.

The main result of this note is of this second type. It consists of formulas for 
the Euclidean distance degree of an {\em arbitrary nonsingular subvariety\/} 
of projective space; different versions are presented, in terms of different
types of information that may be available on $X$. The simplest form of the 
result is the following:
\begin{theorem*}
Let $X$ be a smooth subvariety of $\Pbb^{n-1}$, and assume $X\not\subseteq Q$. Then 
\begin{equation}\label{eq:Eulerin}
\Edd(X) = (-1)^{\dim X} \chi\big(X\smallsetminus (Q\cup H)\big)
\end{equation}
where $H$ is a general hyperplane.
\end{theorem*}
Here $\chi$ is the ordinary topological Euler characteristic. This statement will be
proven in~\S\ref{s:Euler}; in each of~\S\S\ref{s:SH}--\ref{s:CSM} 
we obtain an equivalent formulation of the same result. These serve as stepping 
stones in the proof of~\eqref{eq:Eulerin}, and seem of independent interest. 
Theorems~\ref{thm:segref} and~\ref{thm:milnor} will express $\Edd(X)$ as a 
`correction' $\gamma(X)$ of the generic Euclidean distance degree due to the 
singularities of $Q\cap X$. For example, it will be a consequence of 
Theorem~\ref{thm:milnor} that when $Q\cap X$ has {\em isolated\/} singularities, 
then this correction equals the sum of the Milnor numbers of the singularities. 
(If $X$ is a smooth hypersurface of degree~$\ne 2$, the singularities of 
$Q\cap X$ are necessarily isolated, cf.~\S\ref{ss:hyper}.) Theorem~\ref{thm:segref} 
expresses $\gamma(X)$ in terms of the Segre class of the singularity subscheme
of $Q\cap X$; this version of the result is especially amenable to effective
implementation, using available algorithms for the computation of Segre classes 
(\cite{Harris201726}). Theorem~\ref{thm:CSM} relates the Euclidean distance 
degree to {\em Chern-Schwartz-MacPherson\/} classes, an important notion in the
theory of characteristic classes for singular or noncompact varieties. In fact, 
$\Edd(X)$ admits a particularly simple expression, given in~\eqref{eq:CSM2},
in terms of the Chern-Schwartz-MacPherson class of the nonsingular, but noncompact,
variety $X\smallsetminus Q$. Theorem~\ref{thm:Euler}, reproduced above, follows
from this expression.

The progression of results in~\S\S\ref{s:SH}--\ref{s:Euler} is preceded by a general
formula, Theorem~\ref{thm:masterf}, giving the correction term~$\gamma(X)$ for 
essentially arbitrary varieties $X$. Coupled with~\cite[Proposition~2.9]{produa},
this yields a general formula for $\Edd(X)$. This master formula is our main tool
for the applications to nonsingular varieties obtained in the sections that follow;
in principle it could be used in more general situations, but at this stage we do 
not know how to extract a simple statement such as formula~\ref{eq:Eulerin} from
Theorem~\ref{thm:masterf} without posing some nonsingularity hypothesis on $X$.

Refining the techniques used in this paper may yield more general results, but
this is likely to be challenging. Ultimately, the reason why we can obtain simple
statements such as~\eqref{eq:Eulerin} is that Segre classes of singularity subschemes
of hypersurfaces {\em of a nonsingular variety\/} $X$ are well understood. 
In general, singularities of $X$ will themselves contribute to the singularity subscheme 
of $Q\cap X$, even if the intersection of $Q$ and $X$ is (in some suitable sense) 
`transversal'. In fact, for several of our formulas to hold it is only necessary that $X$ 
be nonsingular along $Q\cap X$ (cf.~Remarks~\ref{rem:Xsin} and~\ref{rem:Xsin2}).

The raw form of our result is a standard application of Fulton-MacPherson
intersection theory, modulo one technical difficulty, which we will attempt to 
explain here. Techniques developed in~\cite{MR3451425} express the ED degree
as the degree of a projection map from a certain correspondence in $\Pbb^{n-1}
\times \Pbb^{n-1}$. Applying Fulton-MacPherson's intersection theory, one
obtains a formula for the ED degree involving the Segre class of  an associated 
subscheme~$Z^\Delta_u$ in the conormal space of $X$ (Theorem~\ref{thm:masterf});
this formula holds for arbitrary $X\not\subseteq Q$. In the nonsingular case, the
formula may be recast in terms of the Segre class {\em in $X$\/} of a scheme 
supported on the singular locus of $Q\cap X$. A somewhat surprising complication 
arises here, since this scheme does {\em not\/} coincide with the singularity 
subscheme of $Q\cap X$. However, we can prove (Lemma~\ref{lem:intdep}) 
that the ideal sheaves of the two subschemes have the same {\em integral closure,\/} 
and deduce from this that their Segre classes coincide. This is key to our explicit 
formulas.

This technical difficulty is likely one of the main obstacles in extending the results 
of this paper to the case of more general subvarieties of projective space, by 
analogous techniques. We may venture the guess that a different approach, aiming
at `understanding' \eqref{eq:Eulerin} more directly, without reference to the theory
of characteristic classes of singular varieties,
may be more amenable to generalization. Finding such an
approach would appear to be a natural project.

Preliminaries on the Euclidean distance degree are given in~\S\ref{s:EDD}. 
In~\S\ref{s:gEDdr} we point out that~\eqref{eq:Eulerin}, in its equivalent formulation
\eqref{eq:Euler2}, agrees with $\gEdd(X)$ when $X$ is nonsingular and
meets the isotropic quadric transversally. We find that this observation clarifies
why a formula such as~\eqref{eq:Eulerin} may be expected to hold without
transversality hypotheses. 
It is perhaps natural to conjecture that an analogue replacing ordinary Euler
characteristics in~\eqref{eq:Eulerin} with the degree $\chima$ of the Chern-Mather 
class may hold for arbitrary varieties. Under the transversality hypothesis, an analogue 
of~\eqref{eq:Euler2} does hold for possibly singular varieties, as we show in 
Proposition~\ref{prop:gEddsin}.
The main body of the paper consists of \S\S\ref{s:If}--\ref{s:Euler}
Examples of applications of the results obtained here are given in~\S\ref{s:Ex}.

\smallskip

{\em Acknowledgments.} This work was supported in part by NSA grant 
H98230-16-1-0016.
The authors thank Giorgio Ottaviani for suggesting the example presented 
in~\S\ref{ss:Segre}.

\section{Preliminaries on the Euclidean distance degree}\label{s:EDD}

As recalled in the introduction, the Euclidean distance degree of a variety in
$\Rbb^n$ is the number of critical nonsingular points of the (squared) distance 
function from a general point outside of the variety. We consider the complex 
projective version of this notion: for a subvariety $X\subseteq \Pbb^{n-1} :=\Pbb(\Cbb^n)$, we 
let $\Edd(X)$ be the Euclidean distance degree of the affine cone over $X$ 
in~$\Cbb^n$, that is, the number of critical points of the function
\begin{equation}\label{eq:dist}
(x_1-u_1)^2 +\cdots + (x_n-u_n)^2
\end{equation}
which occur at nonsingular points of the cone over $X$, where $(u_1,\dots,u_n)$
is a general point.

\begin{remark}\label{rem:XinQ}
If $X$ is a subset of the isotropic quadric~$Q$ (with equation $x_1^2+\cdots 
+ x_n^2=0$), then the quadratic term in~\eqref{eq:dist} vanishes, and
\eqref{eq:dist} has {\em no\/} critical points. Therefore, $\Edd(X)=0$ in this
case, and we can adopt the blanket convention that $X\not\subseteq Q$. 
With suitable positions, our results will hold without this assumption 
(cf.~e.g., Remark~\ref{rem:segreXinQ}). 
\qede\end{remark}

The definition of $\Edd(X)$ may be interpreted in terms of a {\em projective ED 
correspondence,\/} and this will be needed for our results. Our reference here is 
\cite[\S5]{MR3451425} (but we use slightly different notation).
Consider the projective space $\Pbb^{n-1}$ and its dual~$\cPbb^{n-1}$, parametrizing
hyperplanes in $\Pbb^{n-1}$. It is well-known that the projective cotangent space 
$\Tbb^*\Pbb^n:=\Pbb(T^*\Pbb^{n-1})$ of $\Pbb^{n-1}$ may be realized as the incidence 
correspondence  $I\subseteq \Pbb^{n-1}\times \cPbb^{n-1}$ consisting of pairs 
$(p,H)$ with $p\in H$. Every subvariety $X\subsetneq \Pbb^{n-1}$ has a {\em (projective) 
conormal space\/} $\Tbb^*_X\Pbb^{n-1}$, defined as the closure of the projective
conormal variety to $X$; this may be realized as
\[
\Tbb^*_X\Pbb^{n-1} := \overline{
\{ (p,H) \,|\, \text{$p\in X^{ns}$ and $T_pX\subseteq H$} \}
} \subseteq I=\Tbb^*\Pbb^{n-1}\quad.
\]
Consider the subvariety $Z\subseteq \Pbb^{n-1}\times \cPbb^{n-1}\times \Cbb^n$
obtained as the image of
\[
\{(x,y,u)\in (\Cbb^n\smallsetminus \{0\})^2\times \Cbb^n \,|\, u=x+y\}\quad;
\]
that is, 
\[
Z=\{([x],[y],u)\in \Pbb^{n-1}\times \cPbb^{n-1}\times \Cbb^n\,|\,
\dim \langle x,y,u\rangle \le 2\}
\]
consists of points $([x],[y],u)$ such that $[x],[y],[u]$ are collinear. The {\em (projective joint) 
ED correspondence\/} $\cP\cE$ (denoted $\cP\cE_{X,Y}$ in~\cite{MR3451425}) is the 
component of $(\Tbb^*_X\Pbb^{n-1}\times \Cbb^n)\cap Z$ dominating $\Tbb^*_X\Pbb^{n-1}$.
Thus, the fiber of $\cP\cE$ over $([x],[y])\in \Tbb^*_X\Pbb^{n-1}$ consists, for $[x]\ne
[y]$, of the vectors $u\in \Cbb^n$ in the span of $x$ and $y$; this confirms that $\cP\cE$ 
is irreducible and has dimension~$n$. (Since $X\not\subseteq Q$ by our blanket assumption, 
there exist points $([x],[y])\in \Tbb^*_X\Pbb^{n-1}$ with $[x]\ne [y]$.) The projection 
$\cP\cE \to \Cbb^n$ is in fact dominant, and we have the following result.

\begin{lemma}\label{lem:EDcor}
The Euclidean distance degree $\Edd(X)$ equals the degree of the projection
$\cP\cE \to \Cbb^n$.
\end{lemma}

\begin{proof}
This is implied by the argument in the proof of \cite[Proposition~5.4]{MR3451425}.
\end{proof}

Lemma~\ref{lem:EDcor} suggests that one should be able to express $\Edd(X)$ 
in terms of an intersection with the fiber $Z_u$ of $Z$ over a general point $u\in \Cbb^n$.
We may view $Z_u$ as a subvariety of $\Pbb^{n-1}\times \cPbb^{n-1}$:
\[
Z_u = \{ ([x],[y])\in \Pbb^{n-1}\times \cPbb^{n-1}\,|\, \dim \langle x,y,u\rangle \le  2\}\quad,
\]
where $u$ is now fixed (and general). It is easy to verify that $Z_u$ is an $n$-dimensional
irreducible variety, and that
\begin{equation}\label{eq:Zuclass}
[Z_u]=h^{n-2}+h^{n-3} \ch + \cdots + \ch^{n-2}
\end{equation}
in the Chow group $A_*(\Pbb^{n-1}\times \cPbb^{n-1})$. Here, $h$, resp.,~$\ch$ denote 
the pull-back of the hyperplane class from $\Pbb^{n-1}$, resp.~, $\cPbb^{n-1}$.
(For example, one may verify~\eqref{eq:Zuclass} by intersecting $Z_u$ with suitably
chosen $\Pbb^i\times \cPbb^j$ within $\Pbb^{n-1}\times \cPbb^{n-1}$.)
This implies the following statement, cf.~\cite[Proposition~5.4]{MR3451425}.

\begin{lemma}\label{eq:Zuint}
For all $u\in \Cbb^{n-1}$ and all subvarieties $X\subseteq \Pbb^{n-1}$,
\[
Z_u\cdot \Tbb^*_X\Pbb^{n-1} = \sum_{i=0}^{n-2} \delta_i(X)\quad,
\]
where the numbers $\delta_i(X)$ are the {\em polar degrees\/} of $X$.
\end{lemma}

Indeed, the polar degrees are (by definition) the coefficients of the monomials
$h^{n-1-i} \ch^{i+1}$ in the class $[\Tbb^*_X\Pbb^{n-1}]$. 

In view of Lemma~\ref{eq:Zuint}, we define
\[
\gEdd(X):=\sum_{i=0}^{n-2} \delta_i(X)\quad,
\]
the `generic Euclidean distance degree' of $X$. 
By Lemma~\ref{lem:EDcor}, $\Edd(X)$ is the contribution of the projective ED
correspondence to the intersection number $\gEdd(X)=Z_u\cdot \Tbb^*_X\Pbb^{n-1}$
calculated in Lemma~\ref{eq:Zuint}. It is a consequence of~\cite[Proposition~5.4]{MR3451425} 
that if $X$ is in sufficiently general position, then this contribution in fact equals the whole 
intersection number, i.e., $\Edd(X) =\gEdd(X)$.
Our goal is to determine a precise `correction term' evaluating the discrepancy between
$\Edd(X)$ and $\gEdd(X)$ without any {\em a priori\/} hypothesis on $X$.
In \S\ref{s:If} we will formalize this goal and deduce a general 
formula for $\Edd(X)$
for an arbitrary variety $X$. In~\S\S\ref{s:If}--\ref{s:Euler} we will use this
result to obtain more explicit formulas for $\Edd(X)$ under the assumption that $X$
is nonsingular. 

\section{The generic Euclidean distance degree, revisited}\label{s:gEDdr}
This section is not used in the sections that follow, but should help motivating
formula~\eqref{eq:Eulerin}, which will be proven in~\S\ref{s:Euler}. We also
propose a possible conjectural generalization of this formula to arbitrary
projective varieties.

We have defined the `generic' Euclidean distance degree of a subvariety $X\subseteq
\Pbb^{n-1}$ as the sum of its polar degrees. In~\cite[Theorem~5.8]{MR3451425}
it is shown that if $X$ is nonsingular, then
\begin{equation}\label{eq:gEdd}
\gEdd(X)=\sum_{j=0}^{\dim X} (-1)^{\dim X+j} c(X)_j (2^{j+1}-1)\quad;
\end{equation}
this number may be interpreted as the Euclidean distance degree of a general translation 
of~$X$, which will meet $Q$ transversally by Bertini's theorem.
Here $c(X)_j$ is the degree of the component of dimension~$j$ 
of the Chern class $c(TX)\cap [X]$ of $X$. Formula~\eqref{eq:gEdd} holds for 
arbitrarily singular varieties $X$, if one replaces $c(X)$ with the {\em Chern-Mather\/} 
class $\cma(X)$ of $X$ (\cite[Proposition 2.9]{produa}).

Assume first that $X$ is nonsingular.
As a preliminary observation, the reader is invited to perform the following calculus
exercise:
\[
\text{\em For $0\le j\le N$, the coefficient of $t^N$ in the expansion of 
$\frac{t^{N-j}}{(1+t)(1+2t)}$ is 
$(-1)^j (2^{j+1}-1)$.}
\] 
With this understood, we have the following computation:
\begin{align*}
\gEdd(X) &= (-1)^{\dim X}\sum_{j=0}^{\dim X} c(X)_j (-1)^j (2^{j+1}-1) \\
&= (-1)^{\dim X}\sum_{j=0}^{\dim X} c(X)_j \int \frac{h^{\dim X-j}}{(1+h)(1+2h)} \cdot
h^{\codim X}\cap [\Pbb^{n-1}] \\
&= (-1)^{\dim X}\int \frac{1}{(1+h)(1+2h)}\cdot c(TX)\cap [X] \\
&= (-1)^{\dim X}\int \left(1-\frac{h}{1+h} -\frac{2h}{1+2h} + \frac{h\cdot 2h}{(1+h)(1+2h)}\right)
\cdot c(TX)\cap [X]\quad.
\end{align*}
Assuming further that $X$ is transversal to $Q$ and that $H$ is a general hyperplane,
the last expression may be rewritten as
\begin{multline*}
(-1)^{\dim X}\left(\int c(TX)\cap [X] - \int c(T(X\cap H))\cap [X\cap H] \right.\\
\left.- \int c(T(X\cap Q))\cap [X\cap Q] + \int c(T(X\cap Q\cap H))\cap [X\cap Q\cap H] 
\right)
\end{multline*}
(by transversality, all of the loci appearing in this expression are nonsingular). 
The degree of the zero-dimensional component of the Chern 
class of a compact complex nonsingular variety is its topological Euler characteristic, so this 
computation shows that
\begin{equation}\label{eq:gEddEu}
\Edd(X) = (-1)^{\dim X}\big(\chi(X)-\chi(X\cap H)-\chi(X\cap Q) + \chi(X\cap Q\cap H)\big)\quad,
\end{equation}
{\em if $X$ is nonsingular and meets $Q$ transversally\/}
(and where $H$ is a general hyperplane).

Theorem~\ref{thm:Euler} will amount to the assertion that~\eqref{eq:gEddEu} holds 
as soon as $X$ is nonsingular, without any hypothesis on the intersection of $Q$ 
and $X$. By the good inclusion-exclusion properties of the Euler characteristic,
\eqref{eq:gEddEu} is equivalent to \eqref{eq:Eulerin}.

While the computation deriving~\eqref{eq:gEddEu} 
from~\cite[Theorem~5.8]{MR3451425} is trivial under the transversality hypothesis, 
we do not know of any simple way 
to obtain this formula in the general case. The next several sections 
(\S\ref{s:If}--\ref{s:Euler}) will lead to a proof of~\eqref{eq:gEddEu} for arbitrary 
nonsingular varieties. 

The above computation can be extended to singular projective subvarieties. 
Just as the topological Euler characteristic of a nonsingular variety is the degree of its 
top Chern class, we can define an `Euler-Mather characteristic' of a possibly singular 
variety $V$ by setting
\[
\chima(V):= \int \cma(V)\quad,
\]
the degree of the Chern-Mather class of $V$.
This number is a linear combination of Euler characteristics of strata of $V$, with coefficients
determined by the {\em local Euler obstruction $\Eu$,\/} a well-studied numerical invariant of
singularities.

\begin{prop}\label{prop:gEddsin}
For {\em any\/} subvariety $X\subseteq\Pbb^{n-1}$ intersecting $Q$ transversally,
\begin{equation}\label{eq:Euma}
\Edd(X) = (-1)^{\dim X}\big(\chima(X) - \chima(X\cap Q) -\chima(X\cap H) 
+ \chima(X\cap Q\cap H)\big)\quad,
\end{equation}
where $H$ is a general hyperplane.
\end{prop}

\begin{proof}
Argue precisely as in the discussion leading to~\eqref{eq:gEddEu}, 
using~\cite[Proposition 2.9]{produa} in place of~\cite[Theorem~5.8]{MR3451425}. 
The only additional ingredient needed for the computation is the fact that if $W$ is a 
nonsingular hypersurface intersecting a variety $V$ transversally, then
\[
\cma(W\cap V) = \frac{W}{1+W}\cap \cma(V)\quad.
\]
For a much stronger result, and a discussion of the precise meaning of `transversality',
we address the reader to~\cite{schurmann:transversality}, particularly Theorem~1.2.
\end{proof}

Of course \eqref{eq:Euma} specializes to~\eqref{eq:gEddEu} when $X$ is nonsingular,
under the transversality hypothesis; but it does not do so in general, because $Q\cap X$  
may be singular even if $X$ is nonsingular, and $\chima(Q\cap X)$ does not necessarily 
agree with $\chi(Q\cap X)$ in that case. 
Therefore, the transversality hypothesis in Proposition~\ref{prop:gEddsin}
is necessary. The Euler-Mather characteristic of the complement,
\[
(-1)^{\dim X} \chima(X\smallsetminus (Q\cup H)) = \int c_*(\Eu_{X\smallsetminus
(Q\cup H)})
\]
(where $c_*$ denotes MacPherson's natural transformation)
may be the most natural candidate as an expression for $\Edd(X)$ for arbitrary
subvarieties $X\subseteq \Pbb^{n-1}$, without smoothness or transversality
hypotheses.


\section{Intersection formula}\label{s:If}

In~\S\ref{s:EDD} we have defined the projective ED correspondence to be one component
of the intersection $(\Tbb^*_X\Pbb^{n-1}\times \Cbb^n)\cap Z$. We next determine
the union of the {\em other\/} irreducible components.
We denote by $\Delta$ the diagonal in $\Pbb^{n-1}\times \cPbb^{n-1}$.
In this section $X\subseteq \Pbb^{n-1}$ is a subvariety (not necessarily smooth),
and $X\not\subseteq Q$ (cf.~Remark~\ref{rem:XinQ}).

\begin{lemma}\label{lem:Zdec}
We have
\[
(\Tbb^*_X\Pbb^{n-1}\times \Cbb^n)\cap Z = \cP\cE\cup Z^\Delta\quad,
\] 
where the support of $Z^\Delta$ equals the support of
$(\Delta\cap \Tbb^*_X\Pbb^{n-1})\times \Cbb^n$.
\end{lemma}

\begin{proof}
Consider the projection $(\Tbb^*_X\Pbb^{n-1}\times \Cbb^n)\cap Z \to \Pbb^{n-1}\times 
\cPbb^{n-1}$. We have already observed in~\S\ref{s:EDD} that the fiber over 
$([x],[y])\in \Tbb^*_X\Pbb^{n-1}$, $([x],[y])\not\in \Delta$, consists of the span $\langle x,y\rangle$
in $\Cbb^n$; it follows that $\cP\cE$ is the only component of the intersection dominating
$\Tbb^*_X\Pbb^{n-1}$. (Again note that since $X\not\subseteq Q$, there are points
$([x],[y])\in \Tbb^*_X$, $([x],[y])\not\in \Delta$.)

We claim that if $([x],[x])\in \Tbb^*_X\Pbb^{n-1}$, then the fiber of
$(\Tbb^*_X\Pbb^{n-1}\times \Cbb^n)\cap Z$ over $([x],[x])$ consists of the whole
space $\Cbb^{n-1}$; the statement follows immediately from this assertion.

Trivially, $([x],[x],u)\in \Tbb^*_X\Pbb^{n-1}\times \Cbb^n$ for all $u$, so we simply need to
verify that $([x],[x],u)\in Z$ for all $u\in \Cbb^n$. But this is clear, since there are points
$([x'],[y'],u)$ with $u\in \langle x', y'\rangle$ and $([x'],[y'])$ arbitrarily close to $([x],[x])$.
\end{proof}

The fact that $Z$ contains $\Delta\times \Cbb^n$ (used in the proof) may also be verified
by observing that equations for $Z$ in $\Pbb^{n-1}\times \cPbb^{n-1}\times \Cbb^n$
are given by the $3\times 3$ minors of the matrix
\begin{equation}\label{eq:Zeq}
\begin{pmatrix}
x_1 & x_2 & \cdots & x_n \\
y_1 & y_2 & \cdots & y_n \\
u_1 & u_2 & \cdots & u_n
\end{pmatrix}
\end{equation}
associated with a point $([x],[y],u)\in \Pbb^{n-1}\times \cPbb^{n-1}\times \Cbb^n$; the
diagonal $\Delta\times \Cbb^n$ obviously satisfies these equations.

Now we fix a general $u\in \Cbb^n$. By Lemma~\ref{lem:EDcor}, the fiber $\cP\cE_u$
consists of $\Edd(X)$ simple points, which will be disjoint from the diagonal for 
general $u$. On the other hand, for all~$u$, the fiber $Z_u$ (when 
viewed as a subvariety of $\Pbb^{n-1}\times \cPbb^{n-1}$) contains $\Delta$. 
We deduce the following consequence of Lemma~\ref{lem:Zdec}.

\begin{corol}\label{cor:intZu}
For a general $u\in \Cbb^n$,
\[
Z_u\cap \Tbb^*_X\Pbb^{n-1} = \text{\{$\Edd(X)$ simple points\}}\sqcup Z_u^\Delta
\]
(as schemes), where the support of $Z_u^\Delta$ agrees with the support of 
$\Delta\cap \Tbb^*_X\Pbb^{n-1}$.
\end{corol}

Taking into account Lemma~\ref{eq:Zuint} we obtain that
\begin{equation}\label{eq:gamma}
\Edd(X) = \gEdd(X) - \gamma(X)\quad,
\end{equation}
where $\gamma(X)$ is the contribution of $Z_u^\Delta$ to the intersection product 
$Z_u\cdot \Tbb^*_X\Pbb^{n-1}$. This `correction term' $\gamma(X)$ does not 
depend on the chosen (general) $u$, and vanishes if $\Delta\cap \Tbb^*_X\Pbb^{n-1} 
=\emptyset$, since in this case $Z_u^\Delta=\emptyset$ by Corollary~\ref{cor:intZu}. 
This special case recovers the statement of~\cite[Proposition~5.4]{MR3451425}, 
and indeed~\eqref{eq:gamma} is essentially implicit in {\em loc.~cit.}. We are interested 
in computable expressions for the correction term $\gamma(X)$. 

We will first obtain the following master formula, through a direct application of 
Fulton-MacPherson intersection theory. The diagonal $\Delta$ is isomorphic to
$\Pbb^{n-1}$, and we denote by~$H$ its hyperplane class, as well as its restrictions.
(Thus, $H$ agrees with the restriction of both $h$ and $\ch$.)

\begin{theorem}\label{thm:masterf}
With notation as above,
\begin{equation}\label{eq:masterf}
\gamma(X) = \int (1+H)^{n-1}\cap s(Z_u^\Delta, \Tbb^*_X\Pbb^{n-1})
\end{equation}
for $u$ general in $\Cbb^n$.
\end{theorem}

Here, $\int$ denotes degree, and $s(-,-)$ is the {\em Segre class,\/} in the sense
of \cite[Chapter~4]{85k:14004}. Segre classes are effectively computable by available 
implementations of algorithms (see e.g.,~\cite{Harris201726}). However, the need
to obtain explicit equations for the scheme $Z_u^\Delta$, and conditions guaranteeing 
that a given $u$ is general enough, limit the direct applicability of Theorem~\ref{thm:masterf}. 
Our task in the next several sections of this paper will be to obtain from~\ref{eq:masterf} 
concrete computational tools, at the price of requiring $X$ to be of a more specific
type---we will assume in the following sections that $X\not\subseteq Q$ is {\em nonsingular,\/}
but otherwise arbitrary.

The proof of Theorem~\ref{thm:masterf} requires some additional information 
on $Z_u$, which we gather next.
As noted in~\S\ref{s:EDD}, $Z_u$ is an irreducible $n$-dimensional subvariety
of $\Pbb^{n-1}\times \cPbb^{n-1}$. Equations for~$Z_u$ in $\Pbb^{n-1}\times 
\cPbb^{n-1}$ are given by the $3\times 3$ minors of the matrix~\eqref{eq:Zeq},
where now $u=(u_1,\dots, u_n)$ is fixed. The diagonal $\Delta$ is a divisor in 
$Z_u$.

\begin{lemma}\label{lem:Zu}
The subvariety $Z_u$ of $\Pbb^{n-1}\times \cPbb^{n-1}$ is a smooth local complete 
intersection at all points $([x],[y])\ne ([u],[u])$. 
\end{lemma}

\begin{proof}
This statement is clearly invariant under a change of coordinates, so we may assume
$u=(1,0,\dots, 0)$. If either $[x]$ or $[y]$ is not $[u]$, we may without loss of generality 
assume that $x_n\ne 0$, and hence $x_n=1$. The ideal of $Z_u$ at this point 
$([x],[y])$ is generated by the $3\times 3$ minors of
\[
\begin{pmatrix}
x_1 & x_2 & \cdots & x_{n-1} & 1 \\
y_1 & y_2 & \cdots & y_{n-1} & y_n \\
1 & 0 & \cdots & 0 & 0
\end{pmatrix}
\]
and among these we find the $n-2$ minors
\[
y_i - x_i y_n\quad, \quad i = 2, \dots, n-1\quad.
\]
Near $([x],[y])$, these generate the ideal of an irreducible smooth complete 
intersection of dimension $n=\dim Z_u$, which must then coincide with $Z_u$
in a neighborhood of $([x],[y])$, giving the statement.
\end{proof}

Denote complements of $\{([u],[u])\}$ by ${}^\circ$. Thus 
$Z_u^\circ = Z_u\smallsetminus \{([u],[u])\}$, $\Delta^\circ = \Delta\smallsetminus 
\{([u],[u])\}$, etc. By Lemma~\ref{lem:Zu}, $Z_u^\circ$ is a local complete intersection 
in $(\Pbb^{n-1}\times \cPbb^{n-1})^\circ$, and we let $N$ be its normal bundle.

\begin{lemma}\label{lem:N}
With notation as above, $c(N)|_{\Delta^\circ} = (1+H)^{n-1}$.
\end{lemma}

\begin{proof}
Consider the rational map
\[
\xymatrix{
\pi : \Pbb^{n-1}\times \cPbb^{n-1} \ar@{-->}[r] & \Pbb^{n-2}\times \Pbb^{n-2}
}
\]
defined by the linear projection from $[u]$ on each factor. Let $U\subseteq
\Pbb^{n-1}\times \cPbb^{n-1}$ be the complement of the union of $\{[u]\}\times \Pbb^{n-1}$
and $\cPbb^{n-1}\times \{[u]\}$; thus, $\pi|_U: U \to  \Pbb^{n-2}\times \Pbb^{n-2}$ is a regular
map, and $U$ contains $\Delta^\circ$. A simple coordinate computation
shows that $Z_u \cap U = \pi|_U^{-1}(\Delta')$, where $\Delta'$ is the diagonal in
$\Pbb^{n-2}\times \Pbb^{n-2}$. It follows that
\[
N|_{Z_u\cap U} = \pi|_U^*(N_{\Delta'} \Pbb^{n-2}\times \Pbb^{n-2})\cong \pi|_U^*(T\Delta')\quad.
\]
Since $\Delta'\cong \Pbb^{n-2}$, $c(T\Delta') = (1+H')^{n-1}$, where $H'$ is the hyperplane
class. The statement follows by observing that the pull-back of $H'$ to $\Delta^\circ$
agrees with the pull-back of $H$. This is the case, since the restriction
$\pi|_{\Delta^\circ}: \Delta^\circ \cong\Pbb^{n-1}\smallsetminus \{u\}\to \Delta'\cong \Pbb^{n-2}$ 
is a linear projection. 
\end{proof}

With these preliminaries out of the way, we can prove Theorem~\ref{thm:masterf}.

\begin{proof}[Proof of Theorem~\ref{thm:masterf}]
Since $[u]$ is general, it may be assumed not to be a point of $X$. This ensures
that $([u],[u])\not\in \Tbb^*_X\Pbb^{n-1}$; in particular
\[
Z_u^\circ\cap \Tbb^*_X\Pbb^{n-1} = Z_u \cap \Tbb^*_X\Pbb^{n-1}\quad.
\]
It follows that, as a class in $A_*(Z_u \cap \Tbb^*_X\Pbb^{n-1})$, the (Fulton-MacPherson) 
intersection product of $\Tbb^*_X\Pbb^{n-1}$ by $Z_u$ on $\Pbb^{n-1}\times \cPbb^{n-1}$ 
equals the intersection product of $\Tbb^*_X\Pbb^{n-1}$ by $Z_u^\circ$ on 
$(\Pbb^{n-1}\times \cPbb^{n-1})^\circ$.

Therefore, we can view $\gamma(X)$ as the contribution of $Z_u^\Delta$ to
$Z_u^\circ \cdot \Tbb^*_X\Pbb^{n-1}$. Consider the fiber diagram
\[
\xymatrix{
Z_u\cap T^*_X\Pbb^{n-1} \ar[r] \ar[d]_g & T^*_X\Pbb^{n-1} \ar[d] \\
Z_u^\circ \ar[r] & (\Pbb^{n-1}\times \cPbb^{n-1})^\circ
}
\]
By~\cite[\S6.1]{85k:14004} (especially Proposition~6.1(a), Example~6.1.1),
this contribution equals
\[
\int c(g|_{Z_u^\Delta}^* N)\cap s(Z_u^\Delta, \Tbb_X^*\Pbb^{n-1})\quad.
\]
where $N=N_{Z_u^\circ} (\Pbb^{n-1}\times \cPbb^{n-1})^\circ$ as above.
Since $Z_u^\Delta$ is supported on a subscheme of $\Delta^\circ$, 
$c(g|_{Z_u^\Delta}^* N)$ equals (the restriction of) $(1+H)^{n-1}$ by 
Lemma~\ref{lem:N}, and the stated formula follows.
\end{proof}

Summarizing, we have proven that
\begin{equation}\label{eq:genX}
\Edd(X) = \gEdd(X)-\int (1+H)^{n-1}\cap s(Z_u^\Delta, \Tbb^*_X\Pbb^{n-1})
\end{equation}
for all subvarieties $X\not\subseteq Q$ of $\Pbb^{n-1}$. (If $X\subseteq Q$, 
then $\Edd(X)=0$, cf.~Remark~\ref{rem:XinQ}.) The quantity $\gEdd(X)$
is invariant under projective translations, and may be computed in terms of
the Chern-Mather class of $X$. The other term records subtle information
concerning the intersection of $X$ and $Q$, by means of the Segre class 
$s(Z_u^\Delta, \Tbb^*_X\Pbb^{n-1})$. We will focus on obtaining alternative 
expressions for this class.


\section{Euclidean distance degree and Segre classes}\label{s:SH}

Now we assume that $X\subseteq \Pbb^{n-1}$ is a {\em smooth\/} closed subvariety.
As recalled in~\S\ref{s:gEDdr}, in this case $\gEdd(X)$ is given by a certain combination 
of the Chern classes of $X$:
\[
\gEdd(X)=(-1)^{\dim X} \sum_{j=0}^{\dim X} (-1)^j c(X)_j (2^{j+1}-1)\quad.
\]
An application of Theorem~\ref{thm:masterf},
obtained in this section, will yield an explicit formula for the correction term $\gamma(X)$
(and hence for $\Edd(X)$). This result has two advantages over Theorem~\ref{thm:masterf}: 
first, the formula will not depend on the choice of a general $u$; second, its ingredients
will allow us to draw a connection with established results in the theory of
characteristic classes for singular varieties, leading to the results presented 
in~\S\S\ref{s:tran}--\ref{s:Euler}.

The main result of this section is the following. Recall that we are denoting by $Q$ 
the isotropic quadric, i.e. the hypersurface of $\Pbb^{n-1}$ with equation 
$\sum_{i=1}^n x_i^2 = 0$. By our blanket assumption that $X$ should not be 
contained in $Q$, we have that $Q\cap X$ is a (possibly singular) hypersurface of 
$X$. We let $J(Q\cap X)$ denote its {\em singularity subscheme,\/} defined locally by 
the partial derivatives of its equation in $X$ (or equivalently by the appropriate Fitting ideal
of the sheaf of differentials of $Q\cap X$). Also, recall that $h$ denotes the hyperplane
class in~$\Pbb^{n-1}$.

\begin{theorem}\label{thm:segref}
Let $X$ be a smooth subvariety of $\Pbb^{n-1}$, and assume $X\not\subseteq Q$. Then 
\begin{equation}\label{eq:segref}
\Edd(X) = \gEdd(X) - \int \frac{(1+2h)\cdot c(T^*X\otimes \cO(2h))}{1+h} \cap s(J(Q\cap X), X)\quad.
\end{equation}
\end{theorem}

The key ingredient in~\eqref{eq:segref} is the Segre class $s(J(Q\cap X), X)$. This may
be effectively computed by using the algorithm for Segre classes described 
in~\cite{Harris201726}. 

\begin{remark}\label{rem:Xsin}
It will be clear from the argument that it is only necessary to require $X$ to be
nonsingular in a neighborhood of $Q\cap X$.
(Of course $X$ must only have isolated singularities in this case.)
Formula~\eqref{eq:segref} holds
as stated in this more general case; $\gEdd(X)$ may be computed using the 
same formula as in the smooth case (that is, \eqref{eq:gEdd}), but using the 
degrees of the component of the {\em Chern-Mather\/} class of $X$
(\cite[Proposition~2.9]{produa}). The hypothesis $X\not\subseteq Q$ is also not
essential, cf.~Remark~\ref{rem:segreXinQ}.
\qede\end{remark}

The proof of Theorem~\ref{thm:segref} will rely on a more careful study of the schemes
$\Delta\cap \Tbb^*_X\Pbb^{n-1}$ and $Z^\Delta_u$ encountered in~\S\ref{s:If}.
In Corollary~\ref{cor:intZu} we have shown that these two schemes have the
same {\em support;\/} here we will prove the much stronger statement that they
have the same {\em Segre class\/} in $\Tbb^*_X\Pbb^{n-1}$.
Since $\Delta\cap \Tbb^*_X\Pbb^{n-1}$ is closely related with $J(Q\cap X)$
(Lemma~\ref{lem:Delta}), this will allow us to recast Theorem~\ref{thm:masterf}
in terms of the Segre class appearing in~\eqref{eq:segref}, by means of a result
of W.~Fulton.

Recall that $\Delta\subseteq Z_u$ (in fact, $\Delta$ is a divisor in $Z_u$); it follows that
\[
\Delta\cap \Tbb^*_X\Pbb^{n-1} \subseteq Z^\Delta_u\quad.
\] 
These two schemes have the same support (Corollary~\ref{cor:intZu}); but they are in 
general different. It is straightforward to identify 
$\Delta\cap \Tbb^*_X\Pbb^{n-1}$ with a subscheme of~$X$.

\begin{lemma}\label{lem:Delta}
Let $\delta: \Pbb^{n-1}\to \Pbb^{n-1}\times \Pbb^{n-1}$ be the diagonal embedding, and 
let $X$ be a smooth subvariety of $\Pbb^{n-1}$. 
Then $J(Q\cap X) =\delta^{-1}(\Tbb^*_X\Pbb^{n-1})$, i.e., $\delta$ maps $J(Q\cap X)$
isomorphically to $\Delta\cap \Tbb^*_X\Pbb^{n-1}$.
\end{lemma}

\begin{proof}
Since $\Tbb^*_X\Pbb^{n-1}\subseteq \Pbb(T^*\Pbb^{n-1})$, we have
\begin{equation}\label{eq:intQ}
\Delta\cap \Tbb^*_X\Pbb^{n-1} = \Delta\cap \Pbb(T^*\Pbb^{n-1})\cap \Tbb^*_X\Pbb^{n-1}
=\Tbb^*_Q\Pbb^{n-1}\cap \Tbb^*_X\Pbb^{n-1}\quad.
\end{equation}
The diagonal $\delta$ restricts to an isomorphism $q:Q\overset\sim\to \Tbb^*_Q\Pbb^{n-1}$.
By~\eqref{eq:intQ}, we have that $\delta^{-1}(\Tbb^*_X\Pbb^{n-1})$ agrees with 
$q^{-1}(\Tbb^*_X\Pbb^{n-1})$, viewed as a subscheme of $\Pbb^{n-1}$.

Now $q^{-1}(\Tbb^*_X\Pbb^{n-1})$ consists of points $[x]$ such that $[x]\in Q\cap X$
and $T_{[x]}Q\supseteq T_{[x]}X$, and these conditions define $J(Q\cap X)$
scheme-theoretically. The statement follows.
\end{proof}

Determining $Z_u^\Delta$ requires more work. We may assume
without loss of generality that $u=(1,0,\dots, 0)$, so that equations for $Z_u^\Delta$ 
are given by the $3\times 3$ minors of
\[
\begin{pmatrix}
x_1 & x_2 & \cdots & x_{n-1} & x_n \\
y_1 & y_2 & \cdots & y_{n-1} & y_n \\
1 & 0 & \cdots & 0 & 0
\end{pmatrix}
\]
(defining $Z_u$) as well as the requirement that $([x],[y])\in \Tbb^*_X\Pbb^{n-1}$. It is in fact useful to
keep in mind that, for $([x],[y])\in \Tbb^*_X\Pbb^{n-1}$, $\Delta\cap\Tbb^*_X\Pbb^{n-1}$
is defined by the $2\times 2$ minors of
\[
\begin{pmatrix}
x_1 & x_2 & \cdots & x_{n-1} & x_n \\
y_1 & y_2 & \cdots & y_{n-1} & y_n \\
\end{pmatrix}
\]
while $Z^\Delta_u$ is defined (near the diagonal) by the $2\times 2$ minors of
\[
\begin{pmatrix}
x_2 & \cdots & x_{n-1} & x_n \\
y_2 & \cdots & y_{n-1} & y_n \\
\end{pmatrix}
\]
Let $\cI_{\Delta\cap \Tbb^*_X\Pbb^{n-1}}\supseteq \cI_{Z^\Delta_u}$ be the corresponding 
ideal sheaves on $\Tbb^*_X\Pbb^{n-1}$.

\begin{lemma}\label{lem:intdep}
The ideal $\cI_{\Delta\cap \Tbb^*_X\Pbb^{n-1}}$ is integral over $\cI_{Z^\Delta_u}$.
Therefore, 
\begin{equation}\label{eq:sese}
s(Z^\Delta_u,\Tbb^*_X\Pbb^{n-1}) =  s(\Delta\cap \Tbb^*_X\Pbb^{n-1},\Tbb^*_X\Pbb^{n-1})
\quad.
\end{equation}
\end{lemma}

\begin{proof}
The second assertion follows from the first, cf.~the proof of~\cite[Lemma~1.2]{MR97b:14057}.
The first assertion may be verified on local analytic charts, so we obtain an analytic
description of $\Tbb^*_X\Pbb^{n-1}$ at a point $([x],[x])$ of the diagonal. 
Again without loss of generality we may let $[x]=(1:0:\cdots:0:i)\in Q\cap X$, and 
assume that the embedding $\iota:X\to \Pbb^{n-1}$ has the following analytic description 
near this point:
\[
\iota:(\us)=(s_2,\dots, s_d) \mapsto (1:s_2:\dots: s_d:\varphi_{d+1}(\us):\dots: \varphi_n(\us))\quad.
\]
Here $\us$ are analytic coordinates for $X$, centered at $0$, and $\varphi_j(0)=0$
for $j=d+1,\dots, n-1$, $\varphi_n(0)=i$.
The tangent space to $X$ at $(\us)$ is cut out by the $n-d$ hyperplanes 
\begin{equation}\label{eq:cotX}
\varphi_{j2} x_2 + \cdots + \varphi_{jd} x_d -x_j = \Phi_j x_1\quad,\quad j=d+1,\dots, n
\end{equation}
where $\varphi_{jk} = \partial \varphi_j/\partial s_k$ and
\[
\Phi_j = \varphi_{j2}s_2 + \cdots + \varphi_{jd} s_d - \varphi_j\quad.
\]
The hyperplanes~\eqref{eq:cotX} span the fiber of $\Tbb^*_X\Pbb^{n-1}$ over the
point $\iota(\us)$. Therefore, $\Delta\cap \Tbb^*_X\Pbb^{n-1}$ is cut out by the $2\times 2$
minors of the matrix
\begin{equation}
\begin{pmatrix}
1 & s_2 & \cdots & s_d & \varphi_{d+1} & \cdots & \varphi_n \\
\sum_j \lambda_j \Phi_j & -\sum_j \lambda_j \varphi_{j2} & \cdots &
-\sum_j \lambda_j \varphi_{jd} & \lambda_{d+1} & \cdots & \lambda_n
\end{pmatrix}
\end{equation}
where $\lambda_{d+1},\dots,\lambda_n$ are homogeneous coordinates in the fibers
of $\Tbb^*_X\Pbb^{n-1}$, while $Z^\Delta_u$ is cut out by the $2\times 2$ minors of
\begin{equation}\label{eq:ZDuinX}
\begin{pmatrix}
s_2 & \cdots & s_d & \varphi_{d+1} & \cdots & \varphi_n \\
-\sum_j \lambda_j \varphi_{j2} & \cdots &
-\sum_j \lambda_j \varphi_{jd} & \lambda_{d+1} & \cdots & \lambda_n
\end{pmatrix}
\end{equation}

The last several minors in both matrices may be used to eliminate the homogeneous 
coordinates $\lambda_j$, giving $\lambda_j \propto \varphi_j$; in other words, we
find that, near $([x],[x])$ both $\Delta\cap \Tbb^*_X\Pbb^{n-1}$ and $Z^\Delta_u$ lie 
in the local analytic section $\sigma: X\to\Tbb^*_X\Pbb^{n-1}$ defined by
\[
\sigma(\us)\quad:\quad
(\lambda_{d+1}:\cdots : \lambda_n) = (\varphi_{d+1}(\us):\cdots: \varphi_n(\us))\quad.
\]
Setting $\lambda_j = \varphi_j$ we obtain from~\eqref{eq:ZDuinX} generators
\begin{equation}\label{eq:genZDu}
s_k +\sum_j \varphi_j \varphi_{jk} \quad,\quad k=2,\dots, d
\end{equation}
for the ideal of $Z^\Delta_u$ in $\sigma(X)$; the same generators, together with
\begin{equation}\label{eq:genDel}
1-\sum_{j=d+1}^n \varphi_j \Phi_j
\end{equation}
give the ideal of $\Delta\cap \Tbb^*_X\Pbb^{n-1}$ in $\sigma(X)$.
It suffices then to verify that \eqref{eq:genDel} is integral over the ideal generated 
by~\eqref{eq:genZDu}.

For this, note that the hypersurface $\iota^{-1}(Q\cap X)$ has equation
\[
G(\us) = 1+s_2^2+\cdots + s_d^2 + \varphi_{d+1}^2+\cdots + \varphi_n^2\quad.
\]
Since $\partial G/\partial s_k=2(s_k + \sum_j \varphi_j \varphi_{jk})$,
the ideal generated by~\eqref{eq:genZDu} is nothing but
\begin{equation}\label{eq:ideZDu}
\left(\frac{\partial G}{\partial s_2},\dots, \frac{\partial G}{\partial s_d} \right)\quad.
\end{equation}
On the other hand, \eqref{eq:genDel} may be written as
\begin{align*}
1-\sum_{j=d+1}^n \varphi_j \Phi_j &=
1-\sum_{j=d+1}^n \varphi_j (\varphi_{j2}s_2 + \cdots + \varphi_{jd} s_d - \varphi_j) \\
&=1-s_2 \left(\sum_j \varphi_j \varphi_{j2}\right) - \cdots -s_d \left(\sum_j \varphi_j \varphi_{jd} \right)
+\varphi_{d+1}^2+\cdots + \varphi_n^2 \\
&\sim 1+s_2^2 +\cdots + s_d^2 +\varphi_{d+1}^2+\cdots + \varphi_n^2 = G(\us)
\end{align*}
modulo~\eqref{eq:genZDu}. Since
$G$ is integral over~\eqref{eq:ideZDu} by~\cite[Corollary~7.2.6]{MR2266432},
this shows that~\eqref{eq:genDel} is integral over~\eqref{eq:genZDu}, as needed.
\end{proof}

\begin{remark}
The argument also shows that the ideal of $\Delta\cap
\Tbb^*_X\Pbb^{n-1}$ in $\sigma(X)$ equals $(G,\partial G/\partial s_2,\dots, \partial G/\partial s_d)$,
that is, the (local analytic) ideal of $J(Q\cap X)$. This confirms the isomorphism
$J(Q\cap X)\cong \Delta\cap \Tbb^*_X\Pbb^{n-1}$ obtained in Lemma~\ref{lem:Delta}.
\qede\end{remark}

\begin{remark}
The smoothness of~$X$ is needed in our argument, since it gives us direct access
to the conormal space $\Tbb^*_X\Pbb^{n-1}$. However, it is reasonable to expect 
that~\eqref{eq:sese} holds without this assumption, and it would be interesting to
establish this equality for more general varieties.
\qede\end{remark}

By Theorem~\ref{thm:masterf} and Lemma~\ref{lem:intdep},
\begin{equation}\label{eq:masterf2}
\gamma(X) = \int (1+H)^{n-1}\cap s(\Delta\cap \Tbb^*_X\Pbb^{n-1}, \Tbb^*_X\Pbb^{n-1})
\end{equation}
if $X$ is nonsingular and not contained in $Q$. We are now ready to prove 
Theorem~\ref{thm:segref}.

\begin{proof}[Proof of Theorem~\ref{thm:segref}]
Our main tools are Lemma~\ref{lem:Delta} and a result of W.~Fulton.
For a closed embedding $V\subseteq M$ of a scheme in a nonsingular variety $M$, 
Fulton proves that the class
\begin{equation}
\cf(V):=c(TM|_V)\cap s(V,M)
\end{equation}
is {\em independent of~$M$;\/} see \cite[Example~4.2.6]{85k:14004}.
We call $\cf(V)$ the `Chern-Fulton class' of~$V$.

By Lemma~\ref{lem:Delta}, the diagonal embedding $\delta: \Pbb^{n-1} \to \Pbb^{n-1}
\times \cPbb^{n-1}$ restricts to an isomorphism
$\delta|_{J(Q\cap X)}: J(Q\cap X) \overset\sim \to \Delta\cap \Tbb^*_X\Pbb^{n-1}$. Let 
$\pi': \Delta\cap \Tbb^*_X\Pbb^{n-1} \to J(Q\cap X)$ be the natural projection,
that is, the inverse of $\delta|_{J(Q\cap X)}$. Then
\begin{equation}\label{eq:cF}
\cf( \Delta\cap \Tbb^*_X\Pbb^{n-1}) = {\pi'}^* \cf(J(Q\cap X))
\end{equation}
by Fulton's result. We proceed to determine
this class. The Euler sequence for the projective bundle $\Tbb^*_X\Pbb^{n-1}
=\Pbb(T^*_X\Pbb^{n-1})\overset\pi\longrightarrow X$:
\[
\xymatrix{
0 \ar[r] & \cO \ar[r] & \pi^* T^*_X\Pbb^{n-1}\otimes \cO(1) \ar[r] & T(\Tbb^*_X\Pbb^{n-1})
\ar[r] & \pi^* TX \ar[r] & 0
}
\]
yields
\[
c(T(\Tbb^*_X\Pbb^{n-1})) = c(\pi^*T^*_X\Pbb^{n-1}\otimes \cO(1))\cdot \pi^*c(TX)\quad.
\]
Pulling back and tensoring by $\cO(1)$ the cotangent sequence defining the conormal
bundle gives the exact sequence
\[
\xymatrix{
0 \ar[r] & \pi^*T^*_X\Pbb^{n-1}\otimes \cO(1) \ar[r] & \pi^*T^*\Pbb^{n-1}\otimes \cO(1)
\ar[r] & \pi^*T^*X\otimes \cO(1) \ar[r] & 0
}
\]
implying
\[
c(T(\Tbb^*_X\Pbb^{n-1})) = \frac{c(\pi^*T^*\Pbb^{n-1}\otimes \cO(1))\cdot \pi^* c(TX)}
{c(\pi^*T^*X\otimes \cO(1))}\quad.
\]
The cotangent bundle $T^*\Pbb^{n-1}$ may be identified with the incidence correspondence
in the product $\Pbb^{n-1}\times \cPbb^{n-1}$, and $\cO(1)=\cO(h+\ch)$ under this identification
(see e.g., \cite[\S2.2]{produa}). Also, $c(T^*\Pbb^{n-1}) = (1-h)^n$. It follows that
\[
c(T(\Tbb^*_X\Pbb^{n-1})) = 
\frac{(1+\ch)^n\cdot \pi^* c(TX)}{(1+h+\ch)\cdot c(\pi^*T^*X\otimes \cO(h+\ch))}\quad.
\]
Now we restrict to the diagonal. As in~\S\ref{s:If}, we denote by $H$ the hyperplane class
in $\Delta\cong \Pbb^{n-1}$ (and its restrictions); note that $H=h\cdot \Delta=\ch\cdot \Delta$.
Therefore
\[
c(T(\Tbb^*_X\Pbb^{n-1})|_{\Delta\cap \Tbb^*_X\Pbb^{n-1}}) = 
\frac{(1+H)^n\cdot {\pi'}^* c(TX)}{(1+2H)\cdot c({\pi'}^*T^*X\otimes \cO(2H))}
\]
where $\pi'$ denotes the projection $\Delta\cap \Tbb^*_X\Pbb^{n-1} \to J(Q\cap X)$ as
above (and we are omitting other evident restrictions). 
Since $H={\pi'}^* h$, the Chern-Fulton class of $\Delta\cap \Tbb^*_X\Pbb^{n-1}$ must be
\[
\cf(\Delta\cap \Tbb^*_X\Pbb^{n-1}) = {\pi'}^*\left(
\frac{(1+h)^n\cdot c(TX)}{(1+2h)\cdot c(T^*X\otimes \cO(2h))}\right)\cap
s(\Delta\cap \Tbb^*_X\Pbb^{n-1} , \Tbb^*_X\Pbb^{n-1})\quad.
\]
Using~\eqref{eq:cF}, this shows that
\begin{align*}
s(\Delta\cap \Tbb^*_X\Pbb^{n-1} , \Tbb^*_X\Pbb^{n-1})
&={\pi'}^*\left(\frac{(1+2h)\cdot c(T^*X\otimes \cO(2h))}{(1+h)^n\cdot c(TX)}\cap \cf(J(Q\cap X))\right) \\
&={\pi'}^*\left(\frac{(1+2h)\cdot c(T^*X\otimes \cO(2h))}{(1+h)^n}\cap s(J(Q\cap X),X)\right) \quad.
\end{align*}
Since ${\pi'}^*=(\delta|_{Q\cap X})_*$ preserves degrees, and $H={\pi'}^*(h)$, 
\eqref{eq:masterf2} gives
\begin{equation}\label{eq:gamma2}
\gamma(X) = \int \frac{(1+2h)\cdot c(T^*X\otimes \cO(2h))}{1+h}\cap s(J(Q\cap X),X)
\end{equation}
and this concludes the proof.
\end{proof}

\begin{remark}\label{rem:QX}
The very definition of the Euclidean distance degree relies on the square-distance function,
$\sum_i (x_i - u_i)^2$, which is not a projective invariant. Therefore, $\Edd(X)$ {\em does\/}
depend on the choice of coordinates in the ambient projective space $\Pbb^{n-1}$.
Formula~\eqref{eq:gamma},
\[
\Edd(X) = \gEdd(X) - \gamma(X)\quad,
\]
expresses the Euclidean distance degree of a variety in terms of a quantity that {\em is\/}
projectively invariant, i.e., $\gEdd(X)$, and a correction term $\gamma(X)$ which is not.
In fact, the coordinate choice determines the isotropic quadric $Q$: $\sum_i x_i^2=0$
is a specific nonsingular quadric in $\Pbb^{n-1}$. Theorem~\ref{thm:segref} prompts us
to define a transparent `projective invariant version' of the Euclidean distance degree, 
for smooth $X$: $\Edd(Q,X)$ could be defined by the right-hand side of~\eqref{eq:segref},
where now $Q$ is {\em any\/} nonsingular quadric in $\Pbb^{n-1}$ not containing~$X$.
(If $X$ is not necessarily smooth, \eqref{eq:masterf} could likewise be used to define
such a notion.)
This number is clearly independent of the choice of coordinates.
What Theorem~\ref{thm:segref} shows is that $\Edd(Q,X)$ equals the Euclidean distance
degree 
of the variety $X$ once homogeneous coordinates $x_1,\dots, x_n$ are chosen 
so that the equation of $Q$ is $\sum_{i=1}^n x_i^2$.
\qede\end{remark}

\begin{remark}\label{rem:segreXinQ}
As pointed out in Remark~\ref{rem:XinQ}, $\Edd(X)=0$ if $X\subseteq Q$.
Theorem~\ref{thm:segref} is compatible with this fact, in the following sense.
if $Q\cap X=X$, it is natural to set $J(Q\cap X)=X$, and hence
$s(J(Q\cap X),X)=s(X,X)=[X]$. The reader can verify (using~\eqref{eq:gEdd}) that
\[
\int \frac{(1+2h)\cdot c(T^*X\otimes \cO(2h))}{1+h} \cap [X] = \gEdd(X)\quad.
\]
Therefore \eqref{eq:segref} reduces to $\Edd(X)=0$ in this case, as expected.
\qede\end{remark}


\section{Euclidean distance degree and Milnor classes}\label{s:tran}
While the formula in Theorem~\ref{thm:segref} is essentially straightforward
to implement, given the algorithm for the computation of Segre classes
in~\cite{Harris201726}, it is fair to say that its `geometric meaning' is not too
transparent. In this section and the following two we use results from the 
theory of characteristic classes of singular varieties to provide versions of 
the formula in terms of notions with a more (and more) direct geometric 
interpretation.

Our first aim is the following result. The {\em Milnor class\/} of a variety $V$ 
is the signed difference
\[
\cM(V) := (-1)^{\dim V-1} \big(\csm(V)-c_F(V)\big)
\]
between its Chern-Fulton class $c_F(V)$ (which we have already encountered
in~\S\ref{s:SH}) and its {\em CSM} (`Chern-Schwartz-MacPherson') {\em class.\/}

We will denote by $\cM(V)_j$ the component of $\cM(V)$ of dimension~$j$.
The Milnor class owes its name to the fact that if $V$ is a hypersurface with at
worst isolated singularities in a compact nonsingular variety, then the degree of 
its Milnor class is the sum of the Milnor numbers of its singularities
(\cite[Example~0.1]{MR2002g:14005}).

The CSM class of a variety $V$ is a `homology' class which agrees with the total
Chern class of the tangent bundle of $V$ when $V$ is nonsingular, and satisfies
a functorial requirement formalized by Deligne and Grothendieck. See~\cite{MR0361141}
for MacPherson's definition (inspired by this functorial requirement), and
\cite{MR35:3707, MR32:1727} for an earlier equivalent definition by Marie-H\'el\`ene
Schwartz (motivated by the problem of extending theorems of Poincar\'e-Hopf type).
An efficient summary of MacPherson's definition (upgraded to the Chow group) may 
be found in~\cite[Example~19.1.7]{85k:14004}. With notation as in this reference 
(or as in~\cite{MR0361141}), our $\csm(V)$ is $c_*(\one_V)$.

As an easy consequence of functoriality, the degree of $\csm(V)$ equals $\chi(V)$, 
the topological Euler characteristic of $V$. In fact the degrees of all the terms in 
$\csm(V)$ may be interpreted in terms of Euler characteristics 
(\cite[Theorem~1.1]{MR3031565}), and this will be key for the version of the
result we will present in~\S\ref{s:Euler}.

If $V$ is a hypersurface, then $\cf(V)$ equals the class of the {\em virtual\/} tangent
bundle of $V$; it may be interpreted as the limit of the Chern class of a smoothing of
$V$ in the same linear equivalence class. The terms in $\cf(V)$ may therefore also
be interpreted in terms of Euler characteristics (of smoothings of $V$). Roughly,
the Milnor class measures the changes in the Euler characteristics of general hyperplane 
sections of $V$ as we smooth it within its linear equivalence class.

\begin{theorem}\label{thm:milnor}
Let $X$ be a smooth subvariety of $\Pbb^{n-1}$, and assume $X\not\subseteq Q$. Then 
\begin{equation}\label{eq:milnor}
\Edd(X) = \gEdd(X) - \sum_{j\ge 0} (-1)^j \deg\cM(Q\cap X)_j\quad.
\end{equation}
\end{theorem}

Milnor classes are also accessible computationally, cf.~\cite[Example~4.7]{MR1956868}.

\begin{remark}\label{rem:Xsin2}
It suffices to require $X$ to be nonsingular in a neighborhood of $Q\cap X$, 
cf.~Remark~\ref{rem:Xsin}.
\qede\end{remark}

\begin{proof}
We begin by recalling an expression relating the Milnor class of a hypersurface~$V$ 
of a nonsingular variety $M$ to the Segre class of its singularity subscheme $J(V)$. 
Letting $\cL=\cO(V)$,
\begin{equation}\label{eq:ccsh}
\cM(V) = (-1)^{\dim M}\frac{c(TM)}{c(\cL)}\cap \big(s(J(V),M)^\vee \otimes_M \cL\big)\quad.
\end{equation}
This is~\cite[Theorem~I.4]{MR2001i:14009}. The notation used in this statement
are as follows (cf.~\cite[\S1.4]{MR2001i:14009} or~\cite[\S2]{MR96d:14004}): if
$A=\sum_{i\ge 0} a^i$ is a rational equivalence class in $V\subseteq M$, where $a^i$
has codimension $i$ in $M$, and $\cL$ is a line bundle on $V$, then
\[
A^\vee = \sum_{i\ge 0} (-1)^i a^i\quad, \quad A\otimes_M \cL = \sum_{i\ge 0} \frac{a^i}{c(\cL)^i}
\]
(note that the codimension is computed in the ambient variety $M$, even if the class
may be defined in the Chow group of the subscheme $V$).

This notation satisfies several properties, for example a basic compatibility with respect
to Chern classes of tensors of vector bundles. One convenient property is given 
in~\cite[Lemma~3.1]{MR3599436}: with notation as above, the term of codimension 
$c$ in $M$ in
\[
c(\cL)^{c-1}\cap (A\otimes_M \cL)
\]
is {\em independent of $\cL$.\/} In particular, 
\[
\int c(\cL)^{\dim M-1} \cap (A\otimes_M \cL)
\]
is independent of $\cL$, and this implies (using \cite[Proposition~2]{MR96d:14004})
\[
\int c(\cL)^{\dim M-1} \cap A =\int A\otimes_M \cL^\vee\quad.
\]
Apply this fact to $\gamma(X)$ (from~\eqref{eq:gamma2}), viewed as 
\[
\gamma(X) = \int (1+2h)^{\dim X-1}\cap \left(
\frac{c(T^*X\otimes \cO(2h))}{(1+h)(1+2h)^{\dim X-2}}\cap s(J(Q\cap X,X))
\right)\quad,
\]
with $M=X$, $\cL=\cO(2h)$. We obtain
\begin{align*}
\gamma(X) &= \int \left(
\frac{c(T^*X\otimes \cO(2h))}{(1+h)(1+2h)^{\dim X-2}}\cap s(J(Q\cap X,X))
\right)\otimes_M \cO(-2h) \\
&\overset != \int 
\frac{c(T^*X)}{(1-h)(1-2h)}\cap 
\big(s(J(Q\cap X,X))\otimes_M \cO(-2h)\big)
\end{align*}
where the equality $\overset !=$ follows by applying~\cite[Proposition~1]{MR96d:14004}.
Since the degree of a class in $X$ is the degree of its component of dimension~$0$,
i.e., codimension $\dim X$, this gives
\begin{align*}
\gamma(X) &= (-1)^{\dim X}\int 
\left(\frac 1{1-h}\cdot \frac{c(T^*X)}{1-2h}\cap \big(s(J(Q\cap X,X))\otimes_M \cO(-2h)\big)\right)^\vee \\
&= (-1)^{\dim X}\int 
\frac 1{1+h}\cdot \frac{c(TX)}{1+2h}\cap \big(s(J(Q\cap X,X))^\vee \otimes_M \cO(2h)\big)\quad.
\end{align*}
Finally, by~\eqref{eq:ccsh} (with $M=X$, $V=Q\cap X$, $\cL=\cO(V)=\cO(2h)$), we get
\[
\gamma(X) = \int \frac 1{1+h}\cap \cM(Q\cap X)\quad,
\]
and this implies~\eqref{eq:milnor}.
\end{proof}

\begin{corol}\label{cor:isol}
If $Q\cap X$ only has isolated singularities $x_i$, then
\[
\Edd(X) = \gEdd(X) - \sum_i \mu(x_i)
\]
where $\mu(-)$ denotes the Milnor number.
\end{corol}

We will see that this is in fact the case for most smooth hypersurfaces (\S\ref{ss:hyper}).


\section{Euclidean distance degree and Chern-Schwartz-MacPherson classes}\label{s:CSM}
Our next aim is to express the Euclidean distance degree of a nonsingular projective
variety directly, rather than in terms of a correction from a `generic' situation. CSM classes
provide a convenient means to do so. The formula presented in~\S\ref{s:Euler} may look
more appealing, but the alternative~\eqref{eq:CSM} presented here, besides being a
necessary intermediate result, is in a sense algorithmically more direct. 

\begin{theorem}\label{thm:CSM}
Let $X$ be a smooth subvariety of $\Pbb^{n-1}$. Then 
\begin{equation}\label{eq:CSM}
\Edd(X) = (-1)^{\dim X} \sum_{j\ge 0} (-1)^j\left(c(X)_j-\csm(Q\cap X)_j\right)\quad.
\end{equation}
\end{theorem}
Here, $\csm(Q\cap X)_j$ denotes the degree of the $j$-dimensional component of
$Q\cap X$. Again, \eqref{eq:CSM} is straightforward to implement given available
algorithms for characteristic classes (for example \cite{MR1956868,MR3484270,
Harris201726}).

\begin{remark}
If $X\subseteq Q$, then $\csm(Q\cap X)= \csm(X)=c(X)$ by the basic normalization 
property of CSM classes, as $X$ is nonsingular. In this case \eqref{eq:CSM} gives 
$\Edd(X)=0$, as it should (cf.~Remark~\ref{rem:XinQ}). Therefore, we will assume 
$X\not\subseteq Q$ in the proof.
\qede\end{remark}

\begin{remark}
In the proof we will use the fact that $c(X)=\cf(X)$ if $X$ is nonsingular. This prevents
a straightforward generalization of the argument to the case in which $X$ is only
required to be nonsingular in a neighborhood of $Q\cap X$.
\qede\end{remark}

\begin{proof}
According to Theorem~\ref{thm:milnor},
\[
\Edd(X) =\gEdd(X)-(-1)^{\dim X} \sum_{j\ge 0}(-1)^j\big(\csm(Q\cap X)_j-\cf(Q\cap X)_j\big)\quad.
\]
By~\eqref{eq:gEdd}, therefore, $\Edd(X)$ equals
\begin{equation}\label{eq:sprea}
(-1)^{\dim X} \sum_{j\ge 0} (-1)^j
\left((2^{j+1}-1) c(X)_j + \cf(Q\cap X)_j - \csm(Q\cap X)_j \right)\quad.
\end{equation}
By definition,
\[
\cf(Q\cap X) = c(TX)\cap s(Q\cap X,X) = \frac{c(TX)\cdot 2h}{1+2h}\cap [X]
\]
(after push-forward to $X$) since $Q\cap X$ is a hypersurface in $X$ and $\cO(Q)=\cO(2h)$
as $Q$ is a quadric. Therefore
\[
\sum_{j\ge 0} (-1)^j \cf(Q\cap X)_j = \int \frac 1{1+h} \frac{2h}{1+2h}c(TX)\cap [X]
=\sum_{j\ge 0} c(X)_j \int \frac {h^{\dim X-j}}{1+h} \frac{2h}{1+2h} \cap [\Pbb^{\dim X}]
\,.
\]
The coefficient of $c(X)_j$ in this expression equals the coefficient of $h^j$ in the
expansion of
\[
\frac {2h}{(1+h)(1+2h)} = \sum_{j\ge 0} (-1)^{j+1} (2^{j+1}-2) h^j\quad.
\]
Therefore \eqref{eq:sprea} gives
\[
\Edd(X) = (-1)^{\dim X} \sum_{j\ge 0} (-1)^j
\left(\big((2^{j+1}-1)-(2^{j+1}-2)\big) c(X)_j - \csm(Q\cap X)_j \right)
\]
and~\eqref{eq:CSM} follows.
\end{proof}

CSM classes may be associated with locally closed sets: if $V$ is a locally closed
set of a variety~$M$, then $\csm(V)=c_*(\one_V)$ is a well-defined class in $A_*M$.
(If $V=\overline V\smallsetminus W$, with $W$ closed, then $\csm(V)=\csm(\overline V)
-\csm(W)$.) This notation allows us to express~\eqref{eq:CSM} in (even) more concise 
terms: if $X$ is a smooth subvariety of $\Pbb^{n-1}$, and $X\not\subseteq Q$, then 
\begin{equation}\label{eq:CSM2}
\Edd(X) = (-1)^{\dim X} \sum_{j\ge 0} (-1)^j \csm(X\smallsetminus Q)_j\quad.
\end{equation}
Indeed, $c(X)=\csm(X)$ since $X$ is nonsingular.

If $Q\cap X$ is (supported on) a simple normal crossing 
divisor,~\eqref{eq:CSM2} admits a particularly simple expression, given in
the corollary that follows. An illustration of this case will be presented in 
\S\ref{ss:Segre}.

\begin{corol}\label{cor:SNC}
Let $X\subseteq \Pbb^{n-1}$ be a smooth subvariety, and assume the support of
$Q\cap X$ is a divisor $D$ with normal crossings and nonsingular components $D_i$,
$i=1,\dots, r$. Then
\[
\Edd(X) = \int \frac{c(T^*X(\log D))}{1-H}\cap [X] = 
\int \frac 1{1-H}\cdot \frac{c(T^*X)}{\prod_i (1-D_i)}\cap [X]\quad.
\]
\end{corol}

\begin{proof}
If $D$ is a simple normal crossing divisor in $X$, then
\begin{equation}\label{eq:compl}
\csm(X\smallsetminus D) = c(TX(-\log D))\cap [X]
\end{equation}
(see~\cite{MR2001d:14008}, or~\cite{MR1893006} (Proposition~15.3)).
Using this fact in~\eqref{eq:CSM2}, the stated formulas follow from simple 
manipulations and the well-known expression for $c(T^*X(\log D))$ when
$D$ is a simple normal crossing divisor.
\end{proof}

Liao has shown that~\eqref{eq:compl} holds as soon as $D$ is a free divisor
that is locally quasi homogeneous (\cite{Liao1}) or more generally with Jacobian
of linear type (\cite{Liao2}). Therefore, $\Edd(X) = \int \frac{c(T^*X(\log D))}{1-H}\cap [X]$
as in Corollary~\ref{cor:SNC} as soon as the support of $Q\cap X$ satisfies these 
less restrictive conditions.


\section{Euclidean distance degree and Euler characteristics}\label{s:Euler}

Finally, we present a version of the main result which makes no use (in its formulation)
of characteristic classes for singular varieties. This is the version given in the introduction.

\begin{theorem}\label{thm:Euler}
Let $X$ be a smooth subvariety of $\Pbb^{n-1}$. Then 
\begin{equation}\label{eq:Euler}
\Edd(X) = (-1)^{\dim X} \chi\big(X\smallsetminus (Q\cup H)\big)
\end{equation}
where $H$ is a general hyperplane.
\end{theorem}

By the inclusion-exclusion property of the topological Euler characteristic, \eqref{eq:Euler}
is equivalent to
\begin{equation}\label{eq:Euler2}
\Edd(X) = (-1)^{\dim X} \left(
\chi(X)-\chi(X\cap Q)-\chi(X\cap H)+\chi(X\cap Q\cap H)
\right)
\end{equation}
which has the advantage of only involving closed subsets of $\Pbb^{n-1}$. 
Any of the aforementioned implementations of algorithms for characteristic classes of
singular varieties includes explicit functions to compute Euler characteristics of
projective schemes from defining homogenous ideals, 
so~\eqref{eq:Euler2} is also essentially trivial to implement. However, despite its
conceptual simplicity, this expression is computationally expensive.

\begin{remark}
As in~\S\ref{s:CSM}, we have to insist that $X$ be smooth; only requiring it to be
nonsingular in a neighborhood of $Q\cap X$ is not enough for the result to hold.
\qede\end{remark}

\begin{proof}
The statement is a consequence of Theorem~\ref{thm:CSM} and of a result
from~\cite{MR3031565}. Collect the degrees of the components of the CSM class of
a locally closed set $V\subseteq\Pbb^N$ into a polynomial: 
\[
\Gamma_V(t)= \sum_{j\ge 0} \csm(V)_j\, t^j\quad;
\]
and collect the signed Euler characteristics of generic linear sections of $V$ into another
polynomial:
\[
\chi_V(t)= \sum_{j\ge 0} (-1)^j \chi(V\cap H_1\cap \cdots \cap H_j)\, t^j
\]
where the $H_i$'s are general hyperplanes. Then according 
to~\cite[Theorem~1.1]{MR3031565} we have
\[
\Gamma_V(t) = \cI(\chi_V)\quad,
\]
where $\cI$ is an explicit involution. It follows from the specific expression of $\cI$ that
\[
\Gamma_V(-1) =\chi_V(0)+\chi_V'(0)
\]
(see the paragraph preceding the statement of Theorem~1.1 in~\cite{MR3031565}).
Therefore
\begin{equation}\label{eq:genEu}
\sum_{j\ge 0} (-1)^j \csm(V)_j = \chi(V)-\chi(V\cap H)
\end{equation}
for every locally closed set $V$ in projective space. 

The statement of the theorem, in the form given in~\eqref{eq:Euler2}, follows by
applying~\eqref{eq:genEu} to~\eqref{eq:CSM}.
\end{proof}


\section{Examples}\label{s:Ex}

\subsection{Computations}
The ingredients needed to implement the main theorems of this text on computer algebra
systems such as \textrm{Sage} \cite{sagemath} or \textrm{Macaulay2} \cite{M2} are 
all available. One can compute Segre classes and Chern-Mather classes via 
\cite{Harris201726} and Chern-Schwartz-MacPherson classes via (for example) any 
of \cite{MR1956868,MR3044490,Jost2013,MR3484270, MR3385954,Harris201726}.
One issue in concrete examples is that computer algebra systems prefer
to work with $\Qbb$-coefficients, and it is often difficult to write the defining equations 
of a variety which is tangent to the isotropic quadric without extending the field of 
coefficients. This difficulty can sometimes be circumvented by a suitable choice of
coordinates; see Remark~\ref{rem:QX}. Also see~\S\ref{ss:surfaces} below for a discussion
of a template situation. 

In many cases, the `standard' algorithm 
of~\cite[Example 2.11]{MR3451425} appears to be at least as fast as the alternatives
obtained by implementing the results presented here. In some examples these
alternatives are faster, particularly if they take advantage of the refinements which
will be presented below. As an illustration, we can apply Proposition~\ref{prop:planec}
(\S\ref{ss:curves}) to compute the ED degrees of plane curves in terms of a generator 
of their homogeneous ideal.
A (non-optimal) implementation of this method in \textrm{Macaulay2} can be coded 
as follows:
\begin{verbatim}
PP2 = QQ[x,y,z]; C = ideal( F )
S = QQ[s,t,i,Degrees=>{{1,0},{1,0},{0,1}}]/(i^2+1)
J = sub(C,{x=>s^2-t^2,y=>2*s*t,z=>i*(s^2+t^2)})
p = (first degrees radical J)#0 -- ignore degree of i
d = degree C
d*(d-2) + p
\end{verbatim}
where $F=F(x,y,z)$ is the defining homogeneous polynomial for the curve. 

For example, trial runs of computations of the ED degrees of Fermat curves
$x^d+y^d+z^d=0$ for all degrees $d=3,\dots, 40$ took an average of 4.5~seconds
using this method (in a more efficient implementation), and 260~seconds by using 
the standard algorithm. However, direct implementations of the general formulas 
presented in this paper do not fare as well.

The interested reader can find the actual code used here, as well as implementations
of the more general formulas at \url{http://github.com/coreysharris/EDD-M2} .
At this stage, the value of the
formulas obtained in Theorems~\ref{thm:segref}--\ref{thm:Euler} 
appears to rest more on their theoretical applications (in examples such as
the ones discussed below in~\S\ref{ss:Segre}) than in the speed of their computer algebra
implementations.

If the variety is known to be transversal to the isotropic quadric, then its Euclidean
distance degree equals the {\em generic\/} Euclidean distance degree. This may
be computed by using the algorithm for Chern(-Mather) classes in~\cite{Harris201726},
often faster than the standard algorithm. For example, let $S$ be a general hyperplane
section of the second Veronese embedding of $\Pbb^3$ in $\Pbb^9$. Then $S$ is
transversal to the isotropic quadric, and the implementation of the algorithm 
in~\cite{Harris201726} computes its Euclidean distance degree (i.e., $36$) in about 
$2$ seconds. The standard algorithm appears to take impractically long on this 
example; one can improve its performance by first projecting $S$ to a general $\Pbb^3$
(this does not affect the Euclidean distance degree,
by~\cite[Corollary~6.1]{MR3451425}), and the computation then takes about a minute. 

\subsection{Quadrics and Spheres}
Let $X\subseteq \Pbb^{n-1}$ be a nonsingular quadric hypersurface. We say that $X$ is
a {\em sphere\/} if it is given by the equation
\[
x_1^2+\cdots + x_{n-1}^2 = c x_n^2
\]
with $c>0$ a real number. It is clear from the definition in terms of critical points of the
distance function that $\Edd(X)=2$ if $X$ is a sphere in $\Pbb^{n-1}$, $n\ge 2$.

We use this example to illustrate some of the formulas obtained in this paper.

First, since $X$ is a degree~$2$ hypersurface in $\Pbb^{n-1}$,
\[
c(TX)=\frac{c(T\Pbb^{n-1}|_X)}{1+2h} = \frac{(1+h)^n}{1+2h}
\]
(where $h$ denotes the hyperplane class and its pull-backs, as in previous setions).
Applying~\eqref{eq:gEdd}, one easily sees that
\[
\gEdd(X) = 2n-2\quad,
\]
while
\[
c(T^*X\otimes \cO(2h))=\frac{(1-h+2h)^n}{(1+2h)(1-2h+2h)}=\frac{(1+h)^n}{1+2h}\quad.
\]
For a sphere $X\subseteq \Pbb^{n-1}$, the intersection $Q\cap X$ consists of a double 
quadric in $\Pbb^{n-2}$, supported on the transversal intersection $X\cap H$ of $X$ with a 
hyperplane. It follows that $J(Q\cap X)=X\cap H$, and therefore
\[
s(J(Q\cap X), X)=\frac {h\cdot [X]}{1+h}\quad.
\] 
According to Theorem~\ref{thm:segref}, the correction term in this case is given by 
\begin{align*}
\int \frac{(1+2h)\cdot c(T^*X\otimes \cO(2h))}{1+h}\cdot s(J(Q\cap X),X) 
&=\int \frac{(1+2h)(1+h)^n}{(1+h)(1+2h)}\cdot \frac {h\cdot 2h}{1+h}\cap [\Pbb^{n-1}] \\
&=\int 2(1+h)^{n-2}\cdot h^2 \cap [\Pbb^{n-1}]\\
&=2(n-2)\quad.
\end{align*}
By Theorem~\ref{eq:segref}, $\Edd(X)=(2n-2)-2(n-2) = 2$, as it should.

From the point of view of Theorem~\ref{thm:Euler}, we should deal with the topological
Euler characteristics of $X$, $X\cap Q$, $X\cap H$, $X\cap Q\cap H$, where $H$
is a general hyperplane (see~\eqref{eq:Euler2}). If $X$ is a sphere, then $X\cap Q$ is 
(supported on) a nonsingular quadric in $\Pbb^{n-2}$; so is $X\cap H$, and
$X\cap Q\cap H$ is a nonsingular quadric in $\Pbb^{n-3}$.
The Euler characteristic of a nonsingular quadric in $\Pbb^N$ is $N+1$ if $N$ is odd,
$N$ if $N$ is even; therefore
\[
\chi(X)-\chi(X\cap Q)-\chi(X\cap H)+\chi(X\cap Q\cap H) =
\begin{cases}
(n-1) -2(n-1) + (n-3) = -2 & \text{$n$ odd} \\
n -2(n-2) + (n-2) = 2 & \text{$n$ even}
\end{cases}
\]
and by Theorem~\ref{thm:Euler}
\[
\Edd(X) = (-1)^{\dim X} \chi(X\smallsetminus (Q\cup H)) = 2
\]
for all $n$, as expected.

\subsection{Hypersurfaces}\label{ss:hyper}
The case of smooth hypersurfaces of degree $\ge 3$ is more constrained than it
may look at first.

\begin{claim}\label{cla:32}
If two smooth hypersurfaces of degree $d_1$, $d_2$ in projective space are tangent
along a positive dimensional algebraic set, then $d_1=d_2$.
\end{claim}

(This is~\cite[Claim~3.2]{MR1819626}.)
It follows that if $X\subset \Pbb^{n-1}$ is a smooth hypersurface of degree $d\ne 2$, 
then the intersection $Q\cap X$ necessarily has isolated singularities. We are then within
the scope of Corollary~\ref{cor:isol}, and we can conclude
\[
\Edd(X) = \gEdd(X) - \sum_i \mu(x_i)
\]
where the sum is over all singularities $x_i$ of $Q\cap X$, and $\mu(-)$ denotes the 
Milnor number.

\subsection{Curves}\label{ss:curves}
Let $C\subseteq \Pbb^{n-1}$ be a nonsingular curve. Then
\begin{equation}\label{eq:curves}
\Edd(C) = d + \#(Q\cap C) - \chi(C)\quad.
\end{equation}
(This follows immediately from Theorem~\ref{thm:Euler}.)

For example, the twisted cubic parametrized by
\[
(s:t) \mapsto (s^3:\sqrt 3 s^2 t : \sqrt 3 s t^2 : t^3)
\]
has $\Edd$ equal to $3$: indeed, it meets the isotropic quadric at the images of
the solutions of $s^6 + 3 s^4 t^2 + 3 s^2 t^4 + t^6=(s^2+t^2)^3=0$, that is, at two 
points. More generally, the Euclidean distance degree of the rational normal curve 
of degree $n-1$ in $\Pbb^{n-1}$ parametrized by
\[
(s:t) \mapsto \left( \sqrt{\binom {n-1}j} s^{n-1-j} t^j\right)_{j=0,\dots,n-1}
\]
is $(n-1)+2-2=n-1$.

For {\em plane\/} curves, \eqref{eq:curves} admits a particularly explicit formulation.

\begin{prop}\label{prop:planec}
Let $C$ be a nonsingular plane curve, defined by an irreducible homogeneous 
polynomial $F(x,y,z)$. Then
\[
\Edd(C) = d(d-2) + R
\]
where $R$ is the number of distinct factors of the polynomial
$F(s^2-t^2,2st,i(s^2+t^2))\in \Cbb[t]$ and $d=\deg F$.
\end{prop}

\begin{proof}
This follows immediately from~\eqref{eq:curves}, after observing that
$d-\chi(C)=d-(2-(d-1)(d-2))=d(d-2)$ and that the isotropic conic $x^2+y^2+z^2=0$
is parametrized by $(s:t) \mapsto (s^2-t^2, 2st, i(s^2+t^2))$.
\end{proof}

For instance, consider the conic $x^2+2y^2+2iyz=0$. Since
\[
(s^2-t^2)^2 + 2(2st)^2 +2i(2st)(i (s^2+t^2)) =(s-t)^4\quad,
\]
we have $R=1$, therefore its Euclidean distance degree is $2\cdot 0+1=1$. 

For another example, the Fermat quintic $C$: $x^5+y^5+z^5=0$ 
has $R=8$ (as Macaulay2 can verify), therefore $\Edd(C) = 5\cdot 3 + 8 = 23$ 
(cf.~\cite[Example~2.5]{MR3451425}). More generally, the Euclidean distance degree
of the Fermat curve $x^d + y^d + z^d=0$ is $d(d-2)+R$, 
where $R$ is the number of distinct factors of the polynomial
\[
(s^2-t^2)^d + (2st)^d + (i(s^2+t^2))^d\quad.
\]
An explicit expression for the Euclidean distance degree of Fermat hypersurfaces
in any dimension may be found in~\cite[Theorem~4]{MR3574523}.

\subsection{Surfaces}\label{ss:surfaces}
According to Theorem~\ref{thm:Euler}, if $S\subseteq\Pbb^{n-1}$ is a smooth degree-$d$
surface, and $C$ is the support of the intersection $Q\cap S$ (which may very well be
singular), then
\[
\Edd(S)=\chi(S)-\chi(S\cap H) -\chi(C) + \deg(C)\quad,
\]
where $H$ is a general hyperplane. If $n-1=3$, then $\chi(S)=d(d^2-4d+6)$ and
$\chi(S\cap H)=3d-d^2$; for $d\ne 2$, $C$ is necessarily reduced (Claim~\ref{cla:32}), 
so $\deg(C)=2d$. In this case ($S\subseteq \Pbb^3$ a smooth surface of degree $d\ne 2$,
or more generally such that $S\cap Q$ is reduced)
\begin{equation}\label{eq:surfP3}
\Edd(S) = d(d^2-4d+6)-(3d-d^2)-\chi(C)+2d = d(d^2-3d+5) -\chi(C)\quad.
\end{equation}
If $C$ is nonsingular, then $\chi(C)=-2d(d-2)$, and $\Edd(S)=\gEdd(S)=d(d^2-d+1)$.

If $S$ is a plane in $\Pbb^3$, tangent to the isotropic quadric $Q$, then $C=Q\cap S$
is a pair of lines, and~\eqref{eq:surfP3} gives $\Edd(S)=0$. But note that the coefficients
of the equation of this plane are necessarily not all real, so the enumerative interpretation
of $\Edd(S)$ as the number of critical points of a `distance' function should be taken
{\it cum grano salis.\/}

Next let $S$ be a Veronese surface in $\Pbb^5$, described parametrically by
\[
(s:t:u) \mapsto (a_1 s^2: a_2 st: a_3 su : a_4 t^2 : a_5 tu : a_6 u^2)
\]
with $a_1\cdots a_6\ne 0$. 
According to Theorem~\ref{thm:Euler},
\[
\Edd(S) = 3-2-\chi(C) +2\deg C= 2\deg C - \chi(C) + 1\quad,
\]
where $C$ is the support of the curve with equation
\begin{equation}\label{eq:vero}
a_1^2 x^4 + a_2^2 x^2 y^2 + a_3^2 x^2 z^2 + a_4 y^4 + a_5 y^2 z^2 + a_6 z^4=0
\end{equation}
in the plane. (The degree of the image of $C$ in $\Pbb^5$ is $2\deg C$.)

For example, if $C$ is a smooth quartic (the `generic' case), then $\chi(C) = -4$ and 
$\Edd(S)=\gEdd(S)=13$. If the rank of the matrix
\begin{equation}\label{eq:veroma}
\begin{pmatrix}
2a_1^2 & a_2^2 & a_3^2 \\
a_2^2 & 2a_4^2 & a_5^2 \\
a_3^2 & a_5^2 & 2a_6^2
\end{pmatrix}
\end{equation}
is~$1$, then~\eqref{eq:vero} is a double (smooth) conic, so that $\chi(C)=\deg(C)=2$ 
and $\Edd(S)=3$. For example, this is the case for
\begin{equation}\label{eq:doublecon}
(s:t:u) \mapsto (s^2: \sqrt 2\, st: \sqrt 2\, su : t^2 : \sqrt 2\, tu : u^2)\quad.
\end{equation}
If the rank of~\eqref{eq:veroma} is~$2$, then~\eqref{eq:vero} factors as a product
\[
(a' x^2 +b'y^2+c'z^2)(a'' x^2 +b''y^2+c''z^2) = 0
\]
and the factors are different and correspond to nonsingular conics. If these conics meet
transversally, then $\Edd(S)=9$; if they are `bitangent', then $\Edd(S)=7$
(use Corollary~\ref{cor:isol}, or again Theorem~\ref{thm:Euler}).
Explicit examples of these two types are
\[
(s:t:u) \mapsto (s^2: \sqrt 3\, st: 2\, su : \sqrt 2\, t^2 : \sqrt 5\, tu : \sqrt 3\,u^2)
\]
and
\[
(s:t:u) \mapsto (s^2: \sqrt 3\, st: \sqrt 2\, su : \sqrt 2\, t^2 : \sqrt 3\, tu : u^2)\quad.
\]

More general Veronese embeddings are considered in~\S\ref{ss:Segre}.

Note that we could equivalently hold the surface $S=X$ fixed, choosing for
example the standard Veronese embedding, parametrized by 
$(s:t:u) \mapsto (s^2: st: su : t^2 : tu : u^2)$
with ideal
\[
(x_1 x_4 - x_2^2 , x_1 x_5 - x_2 x_3, x_1 x_6 - x_3^2,
x_2 x_5 - x_3 x_4, x_2 x_6 - x_3 x_5, x_4 x_6 - x_5^2)
\]
in $\Pbb^5_{(x_1:\cdots :x_6)}$, and consider a more general nonsingular quadric
\[
Q:\, q_1 x_1^2 + q_2 x_2^2 + \cdots + q_6 x_6^2 = 0\quad,
\]
$q_1\cdots q_6\ne 0$,
in place of the isotropic quadric. This corresponds to a change of coordinates
$x_i \mapsto \sqrt{q_i} x_i$; i.e., $q_i=a_i^2$ with notation as above.
The right-hand side $\Edd(Q,X)$ of~\eqref{eq:segref} (or equivalently 
\eqref{eq:milnor}, \eqref{eq:CSM}, \eqref{eq:Euler}) is independent
of the coordinate choice, cf.~Remark~\ref{rem:QX}. For example, choosing
\[
x_1^2+2 x_2^2+2 x_3^2+x_4^2+2 x_5^2+x_6^2 = 0
\]
for $Q$, along with the standard Veronese embedding, is equivalent to choosing
the standard isotropic quadric along with the embedding~\eqref{eq:doublecon}
(and hence $\Edd(Q,X)=3$ in this case).

This observation may be useful in effective computations, since computer algebra
systems prefer to work with $\Qbb$ coefficients.

\subsection{Segre and Segre-Veronese varieties}\label{ss:Segre}
Let $X$ be the image of the usual Segre embedding
\[
\Pbb(\Cbb^{m_1})\times \cdots \times \Pbb(\Cbb^{m_p}) \to 
\Pbb(\Cbb^{m_1}\otimes \cdots \otimes \Cbb^{m_p})\quad,
\]
that is,
$\Pbb^{m_1-1}\times \cdots \times \Pbb^{m_p-1} \to \Pbb^{m_1\cdots m_p -1}$.
This embedding maps a point
\[
\big(
(s_1^1:\cdots: s_{m_1}^1),\dots, (s_1^p:\cdots: s_{m_p}^p)
\big)
\]
to the point in $\Pbb^{m_1\cdots m_p -1}$ whose homogeneous coordinates $(\ux)$
are all the monomials of multidegree $(1,\dots, 1)$ in the variables $\us^1,\dots, \us^p$.
The equation $\sum_i x_i^2=0$ of the isotropic quadric pulls back to 
\[
\big(\sum_i (s_i^1)^2\big)\cdots \big(\sum_i (s_i^p)^2\big)=0\quad.
\]
Let $Q_i$ be the isotropic quadric in the $i$-th factor. Then this shows that
\[
Q\cap X = (Q_1 \times \Pbb^{m_2-1}\times \cdots \times \Pbb^{m_p-1})
\cup\cdots \cup
(\Pbb^{m_1-1}\times \cdots \times \Pbb^{m_{p-1}-1}\times Q_p)\quad.
\]
It follows that $Q\cap X$ is a divisor with normal crossings and nonsingular components.
Denoting by $h_i$ the hyperplane class in the $i$-th factor, the class of the $i$-th
component is~$2h_i$. By Corollary~\ref{cor:SNC}, 
\begin{equation}\label{eq:Segre}
\Edd(X) = \int\frac 1{1-h_1-\cdots -h_p}\cdot \frac{(1-h_1)^{m_1}\cdots (1-h_p)^{m_p}}
{(1-2h_1)\cdots (1-2h_p)}\cap [X]\quad.
\end{equation}
The conclusion is that 
$\Edd(\Pbb^{m_1-1}\times \cdots \times \Pbb^{m_p-1})$
equals the coefficient of $h_1^{m_1-1}\cdots h_p^{m_p-1}$ in the expansion of
\[
\frac 1{1-h_1-\cdots -h_p}\cdot \prod_{i=1}^p \frac{(1-h_i)^{m_i}}{1-2h_i}\quad.
\]
Friedland and Ottaviani obtain a different expression for the same quantity: they 
prove (\cite[Theorem 4]{MR3273677}, cf.~\cite[Theorem 8.1]{MR3451425}) 
that it must equal the coefficient of
$z_1^{m_1-1}\cdots z_p^{m_p-1}$ in the expression
\begin{equation}\label{eq:FO}
\prod_{i=1}^p \frac{\hat z_i^{m_i} - z_i^{m_i}}{\hat z_i -z_i}
\end{equation}
where $\hat z_i = (z_1+\cdots + z_p)-z_i$. These coefficients must be equal, 
since they both compute the Euclidean distance degrees of Segre varieties.
We note that, for example, 
\[
\Edd(\Pbb^2\times \Pbb^8\times \Pbb^{11}\times \Pbb^{13}\times \Pbb^{24})=
1430462027777307645494624
\]
according to {\em both\/} formulas.

The same technique may be used to deal with {\em Segre-Veronese varieties,\/}
obtained by composing a Segre embedding with a product of Veronese embeddings:
{\small
\[
\Pbb(\Cbb^{m_1})\times\cdots\times \Pbb(\Cbb^{m_p})
\to \Pbb(\Sym^{\omega_1}\Cbb^{m_1}) \times \cdots \times \Pbb(\Sym^{\omega_p}\Cbb^{m_p})
\to \Pbb(\Sym^{\omega_1}\Cbb^{m_1}\otimes \cdots \otimes \Sym^{\omega_p}\Cbb^{m_p})\,.
\]}
Using general coordinates for the Veronese embeddings, each $Q_i$ (with notation as above)
restricts to a smooth hypersurface of degree $2\omega_i$, and the resulting hypersurfaces of the
product meet with normal crossings. The hyperplane class restricts to $\omega_1 h_1 +
\dots + \omega_p h_p$, therefore (again by Corollary~\ref{cor:SNC})
the EDdegree of this variety equals the coefficient of 
$h_1^{m_1-1}\cdots h_p^{m_p-1}$ in the expansion of
\begin{equation}\label{eq:SegVerg}
\frac 1{1-\omega_1 h_1-\cdots -\omega_p h_p}\cdot \prod_{i=1}^p 
\frac{(1-h_i)^{m_i}}{1-2\omega_i h_i}\quad.
\end{equation}
Friedland and Ottaviani also consider Segre-Veronese varieties, but they choose 
suitably invariant coordinates in each factor; this is a different problem. 
(For $p=1$, $m_1=\omega_1=2$, this choice of coordinates is given 
by~\eqref{eq:doublecon}.) They
prove (\cite{MR3273677}, \cite[Theorem~8.6]{MR3451425}) that with these
special coordinates the EDdegree is given again by the coefficient 
of $z_1^{m_1-1}\cdots z_p^{m_p-1}$ in~\eqref{eq:FO}, but where now 
$\hat z_i = (\omega_1 z_1+\cdots + \omega_p z_p)-z_i$.
From our point of view, the choice of coordinates affects the restrictions of the
isotropic quadrics $Q_i$ to the factors. With the invariant coordinates used by
Friedland and Ottaviani each $Q_i$ restricts to a multiple {\em quadric,\/} and
this affects the denominator of~\eqref{eq:SegVerg}: the resulting 
EDdegree equals the coefficient of $h_1^{m_1-1}\cdots h_p^{m_p-1}$ in the 
expansion of
\begin{equation}\label{eq:SegreVer}
\frac 1{1-\omega_1 h_1-\cdots -\omega_p h_p}\cdot \prod_{i=1}^p 
\frac{(1-h_i)^{m_i}}{1-2 h_i}\quad.
\end{equation}
Therefore, this coefficient must agree with the one obtained with the 
Friedland-Ottaviani formula.
(It does not seem combinatorially trivial that this should be the case in general; 
it is easy to verify that both formulas yield $((\omega_1-1)^{m_1}-1)/(\omega_1-2)$ 
for $p=1$.)


\bibliographystyle{halpha}
\bibliography{EDDbib}

\end{document}